\documentclass{article}

\usepackage{graphicx}      % include this line if your document contains figures

\usepackage{amssymb}
\usepackage{amsfonts}
\usepackage{amsmath}
\usepackage{fullpage}
\usepackage{graphicx,color}%
\providecommand{\U}[1]{\protect\rule{.1in}{.1in}}
\newtheorem{theorem}{Theorem}[section]
\newtheorem{lemma}[theorem]{Lemma}

\newtheorem{definition}[theorem]{Definition}

\newtheorem{remark}[theorem]{Remark}

\numberwithin{equation}{section}

\newtheorem{conjecture}[theorem]{Conjecture}
\newtheorem{corollary}[theorem]{Corollary}
\newtheorem{proposition}[theorem]{Proposition}
\newtheorem{simplification}[theorem]{Simplification}
\newenvironment{proof}{{\it Proof.}}{\hfill$\blacksquare$}

\newcommand{\FR}[1]{#1}%{\leavevmode\color{red} #1}}
\renewcommand{\Re}{\mathrm{Re}}

\newcommand{\R}{\mathbb{R}}
\newcommand{\eps}{\varepsilon}

\begin{document}
%\begin{frontmatter}

\title{A universal gap for non-spin quantum control systems\thanks{This research is partially supported by the Fondazione Cariparo ``Visiting Scientist'' Program 2018, the Padua University grant SID
2018 ``Controllability, stabilizability and infimum gaps for control
systems'', prot. BIRD 187147 and by the Grant ANR-15-CE40-0018 of the
ANR}} 

\author{Jean-Paul Gauthier\thanks{Universit\'{e} de Toulon, LIS, UMR\ CNRS 7020, Campus de La Garde, 83041, Toulon CEDEX 9, France, {\tt gauthier@univ-tln.fr}}, Francesco Rossi\thanks{Dipartimento di Matematica ``Tullio Levi-Civita'', Universit\`a\  di Padova, Italy, {\tt francesco.rossi@math.unipd.it}}
}

\maketitle

\bibliographystyle{amsplain}

\begin{abstract}                % Abstract of not more than 250 words.
We prove the existence of a universal gap for minimum time controllability of finite dimensional quantum systems, except for some basic representations of spin groups.

This is equivalent to the existence of a gap in the diameter of orbit spaces of the corresponding compact connected Lie group unitary actions on the Hermitian spheres.
\end{abstract}



\section{Introduction}

In this article, we consider finite dimensional quantum control systems \eqref{quantum}, i.e.:%
\begin{align}
\dot{x}  &  =Ax+\sum_{i=1}^{p}B_{i}x\text{ }u_{i}\tag{$\Sigma$}\label{quantum}%quantum}
\end{align}
where $x\in\mathbb{C}^{n}$, $u_{i}\in\mathbb{R}$ and $  A,B_{i}$ are skew-adjoint matrices.

In several relevant applications in quantum control, the drift $A$ represents the ($n$-dimensional) Schr\"{o}dinger dynamics of the system and the $B_{i}$'s are the laser controls. See examples in  \cite{albertini,BGRS,mirrahimi,rouchon}.

  The unit of time is the ``time of the drift'', i.e.\ we assume that
$||A||=1$, where the norm is the standard matrix norm associated with the
Hermitian norm over $\mathbb{C}^{n}$, i.e. $\|X\|=\sqrt{<x,x>}$ with $<x,y>=x'\bar y$.

In quantum dynamics, the state $x$ evolves on the \textbf{unit} (real) sphere
$S^{2n-1}\subset\mathbb{C}^{n}.$ The minimum time $T(\Sigma)$ of the \ system
$(\Sigma)$ is defined as the supremum of the minimum times necessary
to connect two points of the unit sphere $S^{2n-1}$ with trajectories of
$(\Sigma)$, corresponding to arbitrary $L^{\infty}$ controls $u_{i}(t)$. In other terms, 

\begin{eqnarray}
T(\Sigma)&=&\sup_{X,Y\in S^{2n-1}}\inf \{T\geq 0\mbox{~~s.t.~~} \label{e-TS} \gamma(0)=X,\ \gamma(T)=Y, \gamma \mbox{ trajectory of }\Sigma\}.
\end{eqnarray}

If the system \eqref{quantum} without drift is controllable, it is clear that $T(\Sigma)=0$. Othewise, one needs to take into account the time of the drift to ensure controllability of the system. A conjecture, due to Andrei Agrachev \FR{(oral communication)}, is the following:

\begin{conjecture} \label{conj}
There exists a universal gap $\delta>0$ for the minimum time, i.e. whatever the
dimension $n,$ and whatever the quantum system $(\Sigma),$ either
$T(\Sigma)=0,$ or $T(\Sigma)\geq\delta.$
\end{conjecture}

For more details, see Section \ref{s-andrei}. Our goal in this article is to prove a first result towards the complete proof of this conjecture. Before stating the main theorem, we define the diameter of orbit spaces.

Let $G$ be any \textbf{connected} Lie subgroup of $U(n),$ the unitary group
over $\mathbb{C}^{n}.$ The diameter $D(G)$ is the maximum distance in
$\mathbb{C}^{n}$ of two $G$-orbits in $S^{2n-1}$, i.e.$$D(G):=\sup_{X,Y\in S^{2n-1}}\inf_{g,h\in G}\|gX-hY\|.$$

Our result shows that $D(G)$ is universally bounded from below, except for transitive actions (for which $D(G)=0$) and in the following cases, that we call the {\bf BS-actions}.
\FR{\begin{definition}
The action $(G,\hookrightarrow)$ on $S^{2n-1}$ is a basic spinor action (BS-action for short) if some $\tilde G$ is a covering of $G$ with covering homomorphism $\Pi$ and satisfies:
\begin{enumerate}
\item $\tilde G=T\times H$ where $T$ is a torus and $H =Spin(m)$ is a spin group for some $m$;
\item $\tilde \Phi=\Phi\circ \Pi$ is the tensor product representation $\tilde \Phi=\chi\otimes \Phi'$, where $\chi$ is a character of $T$ and $\Phi'$ is a basic spin representation of $H$.
\end {enumerate}
\end{definition}

Following the terminology by Dynkin from his paper \cite{dynkin}, we call basic representations the irreducible representations that are not Cartan products.

We recall that the basic spin representations of $Spin(m)$ are the basic irreducible representations of this group that are not representations of $SO(m)$. If $m=2k+1$, there is a unique such representation and it has dimension $2^k$. If $m=2k$, there are two representations, usually called the half-spin representations, both of dimension $2^{k-1}$. See e.g. \cite[p. 312]{goodman}.}

The main result of the article is the following.
\begin{theorem}
\label{t-main}There is a universal gap $\delta^{\prime}>0$ for the diameter of unitary actions, except for BS-actions,  i.e.: for any Lie group $G$ which action on $S^{2n-1}$ is not a BS-action, it holds
either $D(G)=0$ or $D(G)\geq\delta^{\prime}$.
\end{theorem}

The connection between Agrachev's conjecture and Theorem \ref{t-main} is described in Section \ref{s-andrei}. A similar conjecture relative to isometric group actions on spheres
is proposed in \cite{gorod3}. Other related references are
\cite{dgms,gorod1,gorod2,green,greenwald,mcgow}. There is also a preprint with a proof of a result that is stronger than Theorem \ref{t-main}, but with a completely different and much longer proof. See \cite{gorodski2019}.

\FR{Another series of results can be found in \cite{green,greenwald}. In particular for closed (finite or infinite) groups,  a lower bound $\eps(n)$ is found in \cite{greenwald}, which depends on the dimension $n$.  Recently, in \cite{green}, the lower bound $\sqrt{2}-o(n)$ is proven for finite isometric real group actions. But this does not impact on the problem for infinite groups, see \cite[Note 3]{green}.

\begin{remark}
In all the papers just cited, some lower bounds are smaller than the bound $\sqrt{2-\sqrt{2}}$ that we compute below.  There is no contradiction here: we consider complex unitary actions only.
 \end{remark}}

We also provide a simple estimate of the constants $\delta$ and $\delta'$ in the following.
\begin{corollary} \label{c-racine2}
If the dimension $n$ of \eqref{quantum} is not of the form $2^k$ nor in the set\footnote{We have two basic representations of $E_6$ in each of the dimension 27 and 351.} 
$$\mathcal{E}:=\{1274, 273, 26, 2925, 27, 351, 912, 56, 1539, 27664, 365750, 8645,$$
$$147250, 72912, 24502400, 5121384450, 87587590464, 6696000, 3875\},
$$
then it holds either $D(G)=T(\Sigma)=0$ or $$D(G)\geq \sqrt{2-\sqrt{2}},\qquad \sin\left(\frac{T(\Sigma)}{2}\right)\geq \frac{\sqrt{2-\sqrt{2}}}2.$$
\end{corollary}

Our proof consists of several successive reductions of the problem. First, in
Section \ref{preliminary}, we give first estimates about the diameter. Then, in Section \ref{s-simple} we reduce the problem to basic irreducible representations of $G$,
where $G$ is a classical compact, simple, simply connected Lie group. In Section \ref{s-exspin}, we study basic representations of the classical groups, except for basic spin representations of spin groups. Section \ref{s-proofs} contains the proof of our main results. We collect technical lemmas in Appendix \ref{appendix}.

\section{Estimates for the diameter} \label{preliminary}
In this section, we find first estimates about the diameter $D(G)$. We first observe that, by invariance of translations in Lie groups, it holds 
\begin{equation}
D(G)=\sup_{X,Y\in S^{2n-1}}\inf_{g\in G}\|gX-Y\|.
\label{e-D}
\end{equation}

We first find a universal gap for reducible representations. We then study irreducible representations.

\subsection{Reducible representations}

We first prove the following result.

\FR{\begin{proposition} \label{p-reduc} Let $\Phi$ be a unitary representation of the group $G$, which realification is reducible. Then, the
diameter satisfies  $D(G)\geq \sqrt{2}$.
\end{proposition}
\begin{remark} This also covers the cases of $\Phi$ being a real representation (e.g. adjoint representations), or reducible as a complex representation.
\end{remark}
\begin{proof} Let $H_1,H_2$ be two (real) orthogonal subspaces, being invariant for $\Phi$. Choose $X\in H_1$ and $Y\in H_2$, for which it holds $$|\Phi(g)X-Y|^2=|\Phi(g)X|^2+|Y|^2=2$$ for any $g\in G$.
\end{proof}}

From now on, we will then consider irreducible representations only. We will also use this estimate for real representations in Section \ref{s-spin}, for some basic representations of the groups $B_n,D_n$.

\subsection{Irreducible representations}

In this section, we define a new problem for the pair $(G,\Phi)$, with the goal of estimating the diameter $D(G)$.

Recall that $D(G)$ satisfies \eqref{e-D}. Observe that it holds
$$\|gX-Y\|^2=\|gX\|^2-2\Re<gX,Y>+\|Y\|^2.$$
Then, recalling that $\|gX\|=\|Y\|=1$, it holds
\begin{equation}
D(G)^{2}=2\sup_{X,Y\in S^{2n-1}}\inf_{g\in G}(1-\Re%
<gX,Y>).\label{diam}%
\end{equation}

Define now
\begin{eqnarray*}
R(G):=\inf_{X,Y\in S^{2n-1}}\sup_{g\in G}\operatorname{Re}<gX,Y>,\\
M(G):=\inf_{X,Y\in S^{2n-1}}\sup_{g\in G}%
|<gX,Y>|,
\end{eqnarray*}
It is clear that $M(G)\geq R(G)$ and that 
$$R(G)\leq 1-\frac{(\delta^{\prime})^2}2 \qquad \Leftrightarrow\qquad  D(G)\geq \delta'.$$

Then
\begin{equation}\label{e-MD}
M(G)\leq \lambda<1 \qquad \Rightarrow\qquad D(G)\geq   \sqrt{2(1-\lambda)}.
\end{equation}

\FR{The goal of the paper is to prove that, except for BS-actions, it either holds $D(G)\geq \sqrt{2}$ or $M(G)\leq \lambda<1$ for a universal $\lambda$. Then, formula \eqref{e-MD} provides the proof of Theorem \ref{t-main}.}

\section{Reduction to basic representations of classical compact simple
Lie groups}
\label{s-simple}

In this section, we reduce our problem of finding an upper estimate for $M(G)$ to the case of \FR{non-BS actions} of the classical compact simple Lie groups $A_n,B_n,C_n,D_n$, by a series of successive simplifications. We first consider the following simplification on the structure of $G$.

\FR{\begin{simplification} $G$ is compact and $\Phi$ is irreducible.
\end{simplification}
\begin{proof} We assume that $\Phi$ is irreducible, otherwise we apply Proposition \ref{p-reduc} to find a lower bound for $D(G)$. Recall that $G$ is injectable in the compact group $U(n)$. We now apply the following theorem, due to Weil. Observe that connectedness of $G$ is crucial here.

 \begin{theorem}[\cite{dixm} 16.4.8] \label{t-dix} Let $G$ be a connected, locally compact group. Then $G$ is injectable in a compact group if and only if  $G=\R^p\times K$ with $p\geq 0$ and $K$ a compact group.
\end{theorem}

We then have $G=\R^p\times K$ with $K$ compact. Since the inclusion $G\hookrightarrow U(n)$ is an irreducible unitary representation of a direct product, it is the tensor product of irreducible unitary representations of the components. It is then of the form $\Phi_1\otimes \Phi_2$, where $\Phi_1$ is a unitary character of $\R^p$ and $\Phi_2$ is an irreducible unitary representation of $K$ compact. Observe that characters of $\R^p$ with $p\geq 1$ are never faithful, while the inclusion is. We then have $p=0$ and $G=K$ compact.
\end{proof}
We applied similar arguments for controllability of quantum systems in \cite{BGRS}.}

\begin{simplification} \label{s-prod} $G$ is the direct product of several simple, compact, connected, simply connected Lie groups. \FR{The representation $\Phi$ is irreducible. If $G=Spin(m)$, then $\Phi$ is not a basic spin representation.}
\end{simplification}
\begin{proof} Since $G$ is compact and connected, it is
covered by the group $T\times H$ being the direct product of a torus and a compact connected, simply connected Lie
group, see \cite[Thm 8.1, p. 233]{brocker}.

The unitary irreducible representation $\Phi$ of $G$ then induces the unitary irreducible representation $\tilde \Phi:=\Phi\circ\Pi$ of $\tilde G:=T\times H$, where $\Pi:\tilde G\to G$ is the covering map. It clearly holds $M(\tilde G)=M(G)$. 

\FR{Since $\tilde \Phi$ is irreducible, it is unitarily equivalent to the product of irreducible representations of the components. Irreducible representations of toric components are characters $e^{i\tau}$, thus they play no role for $M(\tilde G)$.} Hence, we can remove the toric component, write the representation $\Phi'(h):=\tilde \Phi(1,h)$ and have $M(H)=M(G)$. Moreover, since $(G,\Phi)$ is not a BS-action, then $\Phi'$ is not a basic spin representation of $H$ whenever $H=Spin(m)$.

Since $H$ is compact and simply connected, then it is semisimple, see \cite[Remark (7.13), p. 229]{brocker}. Hence, it is the direct product of simple groups, se e.g. \cite[Theorem 0.3]{dynkin}.
\end{proof}

\begin{remark}
Under Simplification \ref{s-prod}, we already know that, for each fixed $n$, there exists a lower bound $\delta(n)$ such that, for all $(G,\Phi)$ non-transitively acting on the sphere $S^{2n-1}$, Theorem \ref{t-main} holds with such $\delta(n)$. Indeed, the number of such groups satisfying Simplification \ref{s-prod} is finite, as a consequence of the classification of simple compact Lie groups.
\end{remark}

We now study the case of $G$ being the direct product of at least two simple compact, connected, simply connected Lie groups. We first have the following bound for unitary representations of direct products.
\begin{proposition}
\label{p-tensprod} Let $\Phi$ be a unitary irreducible representation of $G=G_1\times G_2$, with  $G_1,G_2$ semisimple, compact, connected groups. It then holds $M(G)\leq1/\sqrt{2}$.
\end{proposition}
\begin{proof} Since $\Phi$ is irreducible, then it is equivalent to $\Phi_1(G_1)\otimes\Phi_2(G_2)$ with $\Phi_1,\Phi_2$ irreducible.  We then estimate $M(G)=M(G_1\times G_2)$ as follows: let $\Phi_1(G_1),\Phi_2(G_2)$ act on the Hermitian spaces $V,W$, respectively. Since the representations are irreducible and $G_1,G_2$ are semisimple, then $V,W$ have dimension 2 at least.

\FR{Choose an orthonormal pair of vectors $v,v'$ in $V$, and similarly for $w,w'$ in $W$.  Choose $X=v\otimes w\in S(V)\otimes S(W)$ and observe that it holds $\Phi_1\otimes\Phi_2(g_1,g_2)X\in S(V)\otimes S(W)$ for any choice of $g=(g_1,g_2)\in G$.} Choose $Y=\frac{1}{\sqrt{2}%
}(v\otimes w+v'\otimes w')\in S(V\otimes
W)$ and apply Lemma \ref{l-tens} in the Appendix. This gives 
$$|<\Phi_1\otimes\Phi_2(g_1,g_2)X,Y>|\leq1/\sqrt{2}.$$ Since the estimate is independent on $g_1,g_2$, then the result follows.
\end{proof}

We now apply Proposition \ref{p-tensprod} to $G$ satisfying Simplification \ref{s-prod}. If it is the direct product of at least two simple, compact, connected, simply connected Lie groups, then we   find the estimate $M(G)\leq1/\sqrt{2}$.

This leads to the following simplification.

\begin{simplification}  $G$ is a simple, compact, connected, simply connected Lie group, and $\Phi$ is a unitary irreducible representation. \FR{If $G=Spin(m)$, then $\Phi$ is not a basic spin representation.}
\end{simplification}

We now prove the bound for $M(G)$ when $\Phi$ is the Cartan product of irreducible representations.

\begin{proposition}
\label{p-cartanp} Let $\Phi$ be the Cartan product of two unitary irreducible
representations. Then $M(G)\leq 1/{\sqrt{2}}.$
\end{proposition}
\begin{proof} Let $\Phi$ be the Cartan product of  $\Phi_{1}$ and $\Phi_{2}$, each acting on the Hermitian space $V,W$, respectively.
Let $l_{1},l_{2}$ be two unit lowest weight vectors of $\Phi_{1}$ and
$\Phi_{2}$ respectively. Let $h_{1},h_{2}$ be two unit highest weight vectors
of $\Phi_{1}$ and $\Phi_{2}$ respectively.

By Lemma \ref{l-fuchs} in the Appendix, we have that $h_{1}\otimes h_{2}$ is in the irreducible component of $\Phi_{1}%
\otimes\Phi_{2}$ corresponding to $\Phi$ (by definition), but also
$l_{1}\otimes l_{2}$ is in the same irreducible component. Choose $X=h_{1}\otimes h_{2}$ and $Y=1/\sqrt{2}(h_{1}\otimes h_{2}+l_{1}\otimes l_{2})$. 
Since $\Phi(g)X\in S(V)\otimes S(W)$, one can apply Lemma \ref{l-tens} in the Appendix, that implies 
$$|<\Phi(g)X,Y>| \leq1/\sqrt{2}.$$
Pass to the supremum over $g\in G$ to have the result.
\end{proof}

At this step, it remains only to examine the case where $\Phi$ is not a Cartan product. \FR{As recalled above, it is called a basic representation by Dynkin \cite{dynkin}}. Recall from \cite[Supp. 22]{dynkin} that basic representations stand in a natural correspondence with simple roots, thus each group has a finite number of basic representations.

\subsection{Basic representations of exceptional groups} \label{s-exceptional}

\FR{In this section, we study basic representations of exceptional groups. Their dimensions can be computed by Weyl’s dimension formula, see e.g. \cite[Chap. XIII]{cahn}.

Following \cite[Table 30]{dynkin}, it holds:
\begin{itemize}
\item $G_2$ has two basic representations, of dimensions 7 and 14. The basic representation of dimension 7 is the composition of the inclusion $G_2 \hookrightarrow SO(7)$ and the natural real action of $SO(7)$ on $S^6$, see \cite{greenwald}. This representation acts (transitively) on $S^{6}$, but not on the complex sphere $S^{2n-1}=S^{13}$. Since it is a real representation, then we can remove it, due to Proposition \ref{p-reduc}. The basic representation of dimension 14 is the adjoint representation, see \cite[Table 28]{dynkin}. Since it is also real, we can remove it too.
\item $F_4$ has four basic representations, of dimensions 52 (adjoint representation), 1274, 273, 26.
\item $E_6$ has six basic representations, one for each dimension 78 (adjoint representation), 2925 and two for each dimension 27, 351.
\item $E_7$ has seven basic representations, of dimensions 912, 56, 1539, 27664, 365750, 8645, 133  (adjoint representation).
\item $E_8$ has eight basic representations, of dimensions 147250, 248 (adjoint representation), 72912, 24502500, 5121384450, 87587590464, 6696000, 3875.
\end{itemize}

By removing representations that either act transitively or are real (adjoint representations), we have that the set of possible dimensions is (eventually strictly) contained in $\mathcal{E}$ defined in Corollary \ref{c-racine2}. In fact, we have an equality, since each of such representations acts non-transitively on the complex sphere, see \cite{besse}. For each of such representations, it holds $M(G)<1$. Since they are in finite number, the supremum of such $M(G)$ is strictly smaller than 1 too.}

We now reach the following final simplification of our problem.

\begin{simplification} \label{s-main} $G$ is a classical group in the series $A_n,B_n,C_n,D_n$, and $\Phi$ is a basic, unitary, not real, representation. $(G,\Phi)$ acts non-transitively on the corresponding sphere $S^{2m-1}$ and is not a BS-action.
\end{simplification}

\section{Study of classical compact simple Lie groups}
\label{s-exspin}

In this section, we assume that Simplification \ref{s-main} holds and  study the four series of classical groups.
\subsection{The study of $A_{n}$} \label{s-An}

Let $G=A_n$ satisfy Simplification \ref{s-main}, then $G=SU(n+1)$. Denote by $\Psi$ the natural representation of $G=SU(n+1),$ that acts transitively  over $S^{2n+1}\subset \mathbb{C}^{n+1}$, i.e. $\Psi(g)X=gX$. We have the following proposition.
\begin{proposition}  \label{p-an} Let $G=A_n$ and $\Phi$ satisfy Simplification \ref{s-main}. It holds $$M(G)\leq 1/\sqrt{2}.$$
\end{proposition}
\begin{proof} The basic representations of $G$ are the representations
$\Phi=\Psi^{\{k\}}$ acting (unitarily) on the $k^{th}$ exterior power
$\Lambda^{k}(\mathbb{C}^{n+1})$ in the natural way, where $k=1,...,n$. See \cite[Supplement 24]{dynkin}. 

The case $k=1$ corresponds to $\Phi=\Psi$, that acts transitively, hence it does not satisfy Simplification \ref{s-main}.

Assume now $1<k\leq\frac{n+1}{2}$. Choose $X=e_{1}\wedge e_{2}\wedge...\wedge
e_{k}\in S(\Lambda^{k}(\mathbb{C}^{n+1}))$ and $Y=1/\sqrt{2}(e_{1}\wedge e_{2}\wedge
...\wedge e_{k}+e_{k+1}$ $\wedge e_{k+2}\wedge
...\wedge e_{2k})$. Lemma \ref{l-k2} in the Appendix implies 
$$|<\Phi(g)X,Y>|\leq1/\sqrt{2}$$
for all $g\in G$ and $M(\Phi^{\{k\}}(G))\leq1/\sqrt{2}.$

Assume now that it holds $\frac{n+1}{2}<k<n$. Recall that there exists a duality isometry $I:\Lambda^{k'}(\mathbb{C}^{n+1})\to\Lambda^{k}(\mathbb{C}^{n+1})$ with $k'=n+1-k$, explicitly given by $$e_1\wedge\ldots \wedge\widehat{e_{i_1}}\wedge\ldots\wedge\widehat{e_{i_2}}\wedge\ldots\wedge\widehat{e_{i_k}}\wedge\ldots\wedge e_{n+1}\to e_{i_1}\wedge e_{i_2}\wedge\ldots\wedge e_{i_k},$$
where $\widehat{e_i}$ is the removal of the vector $e_i$. This also implies that the dual representation $\Phi'$ acting on $\Lambda^{k'}(\mathbb{C}^{n+1})$ of a basic representation $\Phi$ acting on $\Lambda^{k}(\mathbb{C}^{n+1})$ is a basic representation itself. Then, apply the previous study to $\Phi'$ and find $X',Y'$ satisfying $|<\Phi'(g) X',Y'>|\leq 1/\sqrt{2}$. By the duality isometry, choose $X=I(X'), Y=I(Y')$, both in $\Lambda^{k}(\mathbb{C}^{n+1})$ and observe that it holds
$$|<\Phi(g) X,Y>|=|<\Phi'(g) X',Y'>|.$$ Then, $M(\Phi^{\{k\}}(G))\leq1/\sqrt{2}$ holds in this case too.

With the same duality isometry, one associates $\Psi^{\{n\}}$ to $\Psi$, that acts transitively on $S^{2n+1}$, thus it does not satisfy Simplification \ref{s-main}.
\end{proof}
\subsection{The study of basic non-spin representations of $B_{n},D_{n}$}
\label{s-spin}
In this section, we study the case of $G=B_n$ or $D_n$ satisfying Simplification \ref{s-main}. Since $G$ is a simply connected compact group, it holds $G=Spin(2n+1),Spin(2n)$, that are the (universal) double covers of $SO(2n+1),SO(2n)$.

  The basic representations fall into two classes, see \cite[Supplements 26-27]{dynkin}: those who come from
representations of $SO(2n+1),SO(2n)$ via the covering mapping and the
remaining, that are the 3 series of basic spinor representations (one for
$Spin(2n+1)$ and two for $Spin(2n)).$ 

Since basic representations coming from representations of $SO(m)$ are real, we can apply Proposition \ref{p-reduc} to them.

\subsection{The study of $C_{n}$}\label{s-Cn}

Let $G=C_n$ satisfy Simplification \ref{s-main}. Then, $G=Sp(n)$ and the natural representation $\Psi$ is the standard representation over
$V=\mathbb{C}^{2n}$ of the real compact form $Sp(n)$ of $Sp(2n,\mathbb{C)}%
$, that acts transitively on $S(V).$

All basic representations are $\Phi=\bar{\Psi
}^{\{k\}}$, with $k=1,...,n$, acting over $\Lambda^{k}V$ again in the natural way, but the
$\Psi^{\{k\}}$ being not irreducible, $\bar{\Psi}^{\{k\}}$ is the highest
weight component of $\Psi^{\{k\}}.$ From \cite[Lemma 0.3, p. 360]{dynkin}, we know that inside the highest weight component $\bar{\Phi
}^{\{k\}}$ of $\Phi^{\{k\}},$ there are two vectors of the form $e_{i1}%
\wedge e_{i2}\wedge...\wedge e_{ik\text{ }}$ and $e_{j1}%
\wedge e_{j2}\wedge...\wedge e_{jk\text{ }}$, with $il\neq jr$
for all $l,r$. Therefore, Lemma \ref{l-k2} in the Appendix can be applied to this case too, finding the same estimate $M(G)\leq 1/\sqrt{2}$ for $C_{n}.$

\section{Proof of main results}\label{s-proofs}

In this section, we prove Theorem \ref{t-main} and Corollary \ref{c-racine2}. We first explain the connection between Agrachev's conjecture and Theorem \ref{t-main}, transforming the problem of finding estimates for $T(\Sigma)$ into a problem of representation of Lie groups.

\subsection{The Agrachev's conjecture} \label{s-andrei}
In this section, we connect Conjecture \ref{conj} with estimates on the diameter $D(G)$. We first observe that the two possibilities $T(\Sigma)=0$ or $T(\Sigma)\geq \delta$ are related to controllability of \eqref{quantum}, when the drift $A$ is replaced by the zero matrix. Indeed, consider the system 
\begin{equation} \label{e-A0}
\dot x=\sum_{i=1}^p B_i x\, u_i.
\end{equation}
Define $L=\mathrm{Lie}(B_i)$ the Lie algebra generated by the $B_i$ and $G$ its exponential group. Then, $G$ acts on $S^{2n-1}$ and the inclusion $\Phi:G\hookrightarrow U(n)$ is a unitary representation of $G$.

We now have the two following possibilities: if $G$ acts transitively on $S^{2n-1}$, then the system \eqref{e-A0} is controllable. By linearity of \eqref{e-A0}, it is controllable in arbitrarily small time. The same holds for the original system \eqref{quantum}, thus $T(\Sigma)=0$. If $G$ does not act transitively on $S^{2n-1}$, then by adding $A$, one has two possibilities. On one side, the system can be not controllable, thus $T(\Sigma)=+\infty>\delta$. On the other side, if the system is controllable, then this is a consequence of the fact that one can use the drift $Ax$, that satifies $\|A\|=1$. Take now a trajectory $X(t)$ of \eqref{quantum} and consider the evolution of the distance between orbits measured on the sphere $S^{2n-1}$ with the standard induced Riemannian distance, i.e. $$d(t):=\inf_{g\in G}\|X(0)-gX(t)\|_{S^{2n-1}}.$$ If $\dot X(t)$ is defined, then it holds $\dot X(t)=AX(t)+\sum_{i=1}^pu_i(t)B_i X(t)$. Denote with $$\mathrm{proj}_B(\dot X(t)):=\sum_{i=1}^pu_i(t)B_i X(t),$$ i.e. its component on the space generated by the $B_i$. Then, for almost every $t\in[0,T]$, it holds
\begin{eqnarray*}
\dot d(t)&\FR{\leq} &\lim_{h\to 0} \inf_{g\in G}\frac{\|X(t+h)-gX(t)\|}{h}=\lim_{h\to 0} \inf_{g\in G}\frac{\|e^{\dot X(t) h}X(t)-gX(t)\|}{h}\\
&\leq&\lim_{h\to 0}\frac{\|(e^{h(A X(t)+\mathrm{proj}_B(\dot X(t)))}-e^{h \mathrm{proj}_B(\dot X(t))}) X(t)\|}{h}\\
&=&\lim_{h\to 0}\frac{\|(e^{hA X(t)+o(h)}-\mathrm{Id}) e^{h \mathrm{proj}_B(\dot X(t))}X(t)\|}{h}=\|AX(t)\|\leq 1.
\end{eqnarray*}
This implies that the curve {\bf on the sphere} is covered with velocity less than one, then the travel time is larger than the distance on $S^{2n-1}$, i.e. $d(t)\leq t$. Instead, the distance $D(G)$ is measured in $\mathbb{C}^n$ with straight lines. Elementary geometric considerations then imply
\begin{equation}\label{e-ineq}
\sin\left(\frac{T(\Sigma)}{2}\right)\geq \frac{D(G)}2,
\end{equation} i.e. Theorem \ref{t-main} implies Conjecture \ref{conj} when G is not a BS-action.

\subsection{Proof of Corollary  \ref{c-racine2} and Theorem \ref{t-main}}

In this section, we prove the main results of our paper. We start with the Corollary.

{\it Proof of Corollary  \ref{c-racine2}.} Let $(G,\Phi)$ act non-transitively on $S^{2n-1}$. We have the following cases:
\begin{description}
\item[1] if $\Phi$ is reducible or real, apply Proposition \ref{p-reduc} to find $D(G)\geq \sqrt2$;
\item[2] otherwise, following Section \ref{s-simple} up to Simplification \ref{s-prod}, assume that $G$ is the direct product of several simple, compact, connected, simply connected Lie groups:
\begin{description}
\item[2.1] If $G$ is the direct product of at least two components, apply Proposition \ref{p-tensprod} to find $M(G)\leq 1/\sqrt2$.
\item[2.2] If $G$ is not the product, but $\Phi$ is the Cartan product of irreducible representations, apply Proposition \ref{p-cartanp} to find $M(G)\leq 1/\sqrt2$.
\end{description}
\item[3] We are left with the case of $G$ simple, compact, connected, simply connected Lie group, and $\Phi$ basic, unitary, not real, representation.
\begin{description}
\item[3.1] \FR{If $G$ is a spin group and $\Phi$ is a basic spin representation, then the space dimension satisfies $n=2^k$. Contradiction.}
\item[3.2] If $G$ is one of the exceptional groups $E_{6},E_{7},E_{8}$, $F_{4},G_{2}$, then the representation either acts transitively, or is real, or the space dimension satisfies $n\in \mathcal{E}$. See Section \ref{s-exceptional}. Contradiction.
\item[3.3] The remaining cases are studied in Section \ref{s-exspin}: they all give $M(G)\leq 1/\sqrt2$.
\end{description}
\end{description}

Recall that $M(G)\leq 1/\sqrt2$ implies $D(G)\geq  \sqrt{2-\sqrt{2}}$ from \eqref{e-MD}. This concludes the proof.\hfill $\blacksquare$

{\it Proof of Theorem \ref{t-main}.} This is identical to the proof of Corollary \ref{c-racine2}, except for Step 3, that is replaced as following:
\begin{description}
\item[3] We are left with the case of $G$ simple, compact, connected, simply connected Lie group, and $\Phi$ basic, unitary, not real, representation.% If $G=Spin(m)$, then $\Phi$ is not a basic spin representation.
\begin{description}
\item[3.1] \FR{If $G$ is a spin group, then $\Phi$ is a basic non-spin representation by hypothesis. Then, by Section \ref{s-spin}, $\Phi$ is real and Proposition \ref{p-reduc} gives $D(G)\geq \sqrt2$.}
\item[3.2] The set $(G,\Phi)$ with $G\in \{E_{6},E_{7},E_{8}$, $F_{4},G_{2}\}$ and $\Phi$ basic, non real and acting non transitively is finite. Each of them satisfies $M(G)<1$, then it exists $\lambda'$ such that $M(G)\leq \lambda'<1$ for all such cases.
\item[3.3] \FR{The remaining cases are $G=A_n,C_n$, for which Sections \ref{s-An}-\ref{s-Cn} ensure $M(G)\leq 1/\sqrt2$.}
\end{description}
\end{description}
\FR{Summing up, some cases give $$M(G)\leq \lambda:=\max\{\lambda',1/\sqrt2\},$$ while other cases satisfy} $D(G)\geq\sqrt2$. By \eqref{e-MD}, we finally prove
$$D(G)\geq \delta':=\min\left\{\sqrt{2},\sqrt{2(1-\lambda)}\right\}=\sqrt{2(1-\lambda)}.$$

%\section{Conclusion}
%
%In this article, we provide a first step towards proving the existence of an universal gap for quantum systems. Moreover, we identified clear future goals: one one side, we showed that the existence of the universal gap is reduced to proving its existence for spin representations. On the other side, we proved that such gap can be effectively estimated by numerical computations on the 27 basic representations of exceptional groups.

\appendix
\section{Technical lemmas} \label{appendix}
In this appendix, we collect some technical lemmas that are used in the paper. We first give a general estimate about distances in the tensor product spaces.
\begin{lemma}\label{l-tens} Let $V,W$ be two Hermitian spaces, and $S(V)$, $S(W)$, $S(V\otimes W)$ be the unit spheres of the corresponding spaces. Let $v,v_1,v_2\in S(V)$ with $v_1\perp v_2$, and $w,w_1,w_2\in S(W)$ with $w_1\perp w_2$. It then holds
$$\left|<v\otimes w,1/\sqrt{2}(v_1\otimes w_1+v_2\otimes w_2)>\right|\leq 1/\sqrt2.$$
\end{lemma}
\begin{proof} Complete $v_1,v_2$ to an orthonormal basis of $V$, and similarly for $w_1,w_2,W$. Write now
$$v=x_1v_1+x_2 v_2+\ldots\in S(V),\qquad w=y_1w_2+y_2 w_2+\ldots \in S(V_2).$$
It holds $x_1^2+x_2^2\leq 1$, as well as  $y_1^2+y_2^2\leq 1$. A direct computation then shows
\begin{eqnarray*}
\left|<v\otimes w,1/\sqrt{2}(v_1\otimes w_1+v_2\otimes w_2)>\right|=1/\sqrt{2}|x_{1}y_{1}+x_{2} y_{2}|\leq1/\sqrt{2}.
\end{eqnarray*}
\end{proof}

We now give a similar estimate for wedge products.
\begin{lemma}\label{l-k2}Let $g\in SU(n)$. Take $X=ge_{1}\wedge ge_{2}\wedge...\wedge
ge_{k}\in S(\Lambda^{k}(\mathbb{C}^{n+1}))$ and $Y=1/\sqrt{2}(e_{1}\wedge e_{2}\wedge
...\wedge e_{k}+e_{k+1}\wedge e_{k+2}\wedge
...\wedge e_{2k})$ in $\Lambda^{k}(\mathbb{C}^n)$ with $1<k\leq \frac{n}{2}$. It then holds
$$\left|<X,Y>\right|\leq 1/\sqrt{2}.$$
\end{lemma}
\begin{proof} A direct computation shows that $\sqrt{2}\left|<X,Y>\right|=|\mathrm{det}(A)+\mathrm{det}(B)|$ with $A,B$ being $k\times k$ matrices defined by $A_{ij}=<ge_i,e_j>$ and $B_{ij}=<ge_i,e_{j+k}>$. We then aim to prove that $|\mathrm{det}(A)+\mathrm{det}(B)|\leq 1$.

Denote with $C=[A\ B]$ the matrix composed by placing side by side columns of $A$ and $B$. Since $|ge_i|=1$, then all rows of $C$ are vectors in $\mathbb{C}^{2k}$ with norm less or equal than 1. Consider a block matrix $$\left(\begin{array}{cc}J_1& 0\\0& \mathrm{Id}\end{array}\right),$$ where each block has dimension $k\times k$ and $J\in SU(k)$. It is clear that the induced change of coordinates in $\mathbb{C}^{2k}$ keeps invariant both the length of vectors and the determinants of $A,B$. The same holds for $$\left(\begin{array}{cc}\mathrm{Id}& 0\\0& J_2\end{array}\right).$$

By the Schur decomposition, choosing $J_1,J_2$ properly, we are reduced to the case in which $A,B$ are upper triangular and such that each row of $C$ has norm less or equal than 1. For each $i=1,\ldots, k$ it then holds $|B_{ii}|\leq \sqrt{1- |A_{ii}|^2}$ and $A_{ii}\in[-1,1]$. Thus
\begin{eqnarray*}
&&|\mathrm{det}(A)+\mathrm{det}(B)|\leq |\mathrm{det}(A)|+|\mathrm{det}(B)|\leq \Pi_{i=1}^k |A_{ii}|+\Pi_{i=1}^k \sqrt{1- |A_{ii}|^2}.
\end{eqnarray*}
Since it holds $0\leq | A_{ii}|\leq 1$, maxima and minima of $\Pi_{i=1}^k |A_{ii}|+\Pi_{i=1}^k\sqrt{1- |A_{ii}|^2}$ are reached either on boundaries, in which it takes values in $[0,1]$, or when it holds $|A_{ii}|=\frac{\sqrt{2}}{2}$ for all $i=1,\ldots,k$, thus taking the value $2^{1-\frac{k}2}$. Then, we have the result for $k\geq 2$.
\end{proof}

We end this appendix with a lemma about lowest weight vectors in Cartan products of representations.
\begin{lemma}\label{l-fuchs} Let $\Phi$ be the Cartan product of  $\Phi_{1}$ and $\Phi_{2}$, each acting on the Hermitian space $V,W$, respectively.
Let $l_{1},l_{2}$ be two unit lowest weight vectors of $\Phi_{1}$ and
$\Phi_{2}$ respectively. Let $h_{1},h_{2}$ be two unit highest weight vectors
of $\Phi_{1}$ and $\Phi_{2}$ respectively.

Then, both $h_{1}\otimes h_{2}$ and $l_{1}\otimes l_{2}$ are in the same irreducible component of $\Phi_{1}
\otimes\Phi_{2}$ corresponding to $\Phi$.
\end{lemma}
\begin{proof} The fact that $h_{1}\otimes h_{2}$ is in the irreducible component corresponding to $\Phi$ comes from the definition of Cartan product. We now prove that $l_{1}\otimes l_{2}$ is in the same component.

Consider the opposition element $w_0$ of the Weyl group of $G$, i.e. the element that is the longest among maximally reduced words $w$ in terms of elementary reflections. The operator $w_0$ maps the maximum weight $\lambda_M$ of a representation $\Phi$ to $w_0(\lambda_M)$ being the maximum weight of the contragredient representation. As a consequence, \FR{$-w_0(\lambda_M)$ is the  highest weight of the contragredient representation and} $w_0(\lambda_M)=\lambda_m$ is the minimum weight of $\Phi$, see \cite{fuchs}.

Recall that weights of a representation are stable under the (linear) action of the Weyl group. Thus, given two representations $\Phi_{1},\Phi_{2}$ the corresponding maximal weights $\lambda_M^1,\lambda_M^2$ and minimal weights $\lambda_m^1,\lambda_m^2$, it holds
 $$w_0(\lambda_M^1+\lambda_M^2)=w_0(\lambda_M^1)+w_0(\lambda_M^2)=\lambda_m^1+\lambda_m^2.$$
Thus implies that $\lambda_m^1+\lambda_m^2$ is a weight of the Cartan product $\Phi$. Since $l_1\otimes l_2$ is the only weight vector associated to such weight, then it is a weight vector of $\Phi$.
\end{proof}

 \section*{Acknowledgements}
 \FR{The authors thank Andrei Agrachev for fruitful discussions. They also thank the anonymous reviewer for many useful suggestions.}

\bibliography{biblio}
\end{document}